\documentclass{amsart}
\makeatletter

\@addtoreset{equation}{section}
\makeatother
\usepackage{amsmath,amssymb,amsthm}
\newtheorem{theorem}[equation]{Theorem}
\newtheorem{proposition}[equation]{Proposition}

\newtheorem*{conj*}{Conjecture}
\newtheorem*{theorem*}{Theorem}
 
\newtheorem{corollary}[equation]{Corollary}
\newtheorem{conj}[equation]{Conjecture}
\newtheorem{lemma}[equation]{Lemma}
\begin{document}

\title{Uniform bounds of Piltz divisor problem over number fields}
\author{Wataru Takeda}
\address{Department of Mathematics, Nagoya University, Chikusa-ku, Nagoya 464-8602,
Japan}
\email{d18002r@math.nagoya-u.ac.jp}

\begin{abstract}
We consider the upper bound of Piltz divisor problem over number fields. Piltz divisor problem is known as a generalization of the Dirichlet divisor problem. We deal with this problem over number fields and improve the error term of this function for many cases. Our proof uses the estimate of exponential sums. We also show uniform results for ideal counting function and relatively $r$-prime lattice points as one of applications.
\end{abstract}

\maketitle
\section{Introduction}
The behavior of arithmetic functions has long been studied and it is one of the most important research in analytic number theory.
But many arithmetic functions $f(n)$ fluctuate as $n$ increases and it becomes difficult to deal with them. Thus many authors study partial sums $\sum_{n\le x}f(n)$ to obtain some information about arithmetic functions $f(n)$.
In this paper we consider Piltz divisor function $I_K^m(x)$ over number field.
Let $K$ be a number field with extension degree $[K:\mathbf Q]=n$ and let $\mathcal{O}_K$ be its ring of integers. Let $D_K$ be absolute value of the discriminant of $K$.
Then Piltz divisor function $I_K^m(x)$ counts the number of $m$-tuples of ideals $(\mathfrak{a}_1, \mathfrak{a}_2,\ldots,\mathfrak{a}_m)$ such that product of their ideal norm $\mathfrak{Na}_1\cdots\mathfrak{Na}_m\le x$.
It is known that \begin{equation} \label{res}I_K^m(x)\sim \underset{s=1}{Res}\ \left(\zeta_K(s)^m\frac {{x^s}}s\right).\end{equation}
We denote $\Delta_K^m(x)$ be the error term of $I_K^m(x)$, that is, $I_K^m(x)- \underset{s=1}{Res}\ \left(\zeta_K(s)^m \frac {{x^s}}s\right)$.

The case of $m=1$ this function is the ordinary ideal counting function over $K$. For simplicity we substitute $I_K(x)$ and $\Delta_K(x)$ for $I_K^1(x)$ and $\Delta_K^1(x)$ respectively. 
There are many results about $I_K(x)$ from 1900's. In the case $K=\mathbf Q$, integer ideals of $\mathbf Z$ and positive integers are in one-to-one correspondence, so $I_{\mathbf Q}(x)=[x]$, where $[\cdot]$ is the Gauss symbol. 
For the general case, the best estimate of $\Delta_K(x)$ hitherto is the following theorem:
\begin{theorem}
\label{idealhi}
The following estimates hold. For all $\varepsilon>0$
\begin{center}
\begin{tabular}{cll}
$n=[K:\mathbf Q]$&$\Delta_K(x)$&\rule[-2mm]{0mm}{6mm}\\
\hline
$2$&$O\left(x^{\frac{131}{416}}\left(\log x\right)^{\frac{18627}{8320}}\right)$&Huxley.  \cite{Hu00}\rule[-2mm]{0mm}{6mm}\\
$3$&$O\left(x^{\frac{43}{96}+\varepsilon}\right)$&M\"uller. \cite{mu88}\rule[-2mm]{0mm}{6mm}\\
$4$&$O\left(x^{\frac{41}{72}+\varepsilon}\right)$&Bordell\`es.  \cite{bo15}\rule[-2mm]{0mm}{6mm}\\
$5\le n\le10$&$O\left(x^{1-\frac4{2n+1}+\varepsilon}\right)$&Bordell\`es. \cite{bo15}\rule[-2mm]{0mm}{6mm}\\
$11\le n$&$O\left(x^{1-\frac3{n+6}+\varepsilon}\right)$&Lao. \cite{La10}\rule[-2mm]{0mm}{6mm}\\
\end{tabular}
\end{center}
\end{theorem}
There are also many results about $I_{\mathbf Q}^m$ from 1800's. In 1849 Dirichlet shows that \[I_{\mathbf Q}^2(x)=x\log x+(2\gamma-1)x+O\left(x^{\frac12}\right),\] where $\gamma$ is the Euler constant, defined by the equation\[\gamma=\lim_{n\rightarrow \infty}\left(\sum_{k=1}^n\frac1k-\log n\right).\]
The $O$-term is improved by many researchers many times, the best estimate hitherto is ${x^{\frac{517}{1648}+\varepsilon}}$ \cite{bw17}.

As we have mentioned above, there exists many results about other divisor problems but it seems that there are not many results about piltz divisor problem over number fields. In 1993, Nowak shows the following theorem:
\begin{theorem}[Nowak \cite{No93}]
\label{no}
When $n=[K:\mathbf{Q}]\ge2$, then we get
\[\Delta_K^m(x)=\left\{
\begin{array}{ll}
O_K\left(x^{1-\frac2{mn}+\frac8{mn(5mn+2)}} (\log x)^{m-1-\frac{10(m-2)}{5n+2}}\right)& \text{ for } 3\le mn\le6,\\
O_K\left(x^{1-\frac2{mn}+\frac3{2m^2n^2}} (\log x)^{m-1-\frac{2(m-2)}{mn}}\right)& \text{ for } mn\ge 7.
\end{array}
\right.\]
\end{theorem}

For the estimate of lower bound, Girstmair, K\"uhleitner, M\"uller and Nowak obtain the following $\Omega$-results:

\begin{theorem}[Girstmair, K\"uhleitner, M\"uller and Nowak \cite{gk05}]
For any fixed number field $K$ with $n=[K:\mathbf Q]\ge2$ 
\begin{equation}\label{omega}
\Delta_K^m(x)=\Omega\left(x^{\frac12-\frac{1}{2mn}}(\log x)^{\frac12-\frac{1}{2mn}}(\log\log x)^{\kappa}(\log\log\log x)^{-\lambda}\right),
\end{equation}
where $\kappa$ and $\lambda$ are constants depending on $K$. To be more precise, let $K^{gal}$ be the Galois closure of $K/\mathbf Q$, $G=Gal\left(K^{gal}/\mathbf{Q}\right)$ its Galois group and $H=Gal\left(K^{gal}/K\right)$ the subgroup of $G$ corresponding to $K$. Then \[\kappa=\frac{mn+1}{2mn}\left(\sum_{\nu=1}^{n}\delta_\nu\nu^{\frac{2mn}{mn+1}}-1\right)\text{ and } \lambda=\frac{mn+1}{4mn}R+\frac{mn-1}{2mn},\]
where \[\delta_\nu=\frac{|\{\tau\in G~|~|\{\sigma\in G~|~\tau\in\sigma H\sigma^{-1}\}|=\nu|H|\}|}{|G|} \]
and $R$ is the number of $1\le\nu\le n$ with $\delta_\nu>0$.
\end{theorem}
We know the following conditional result:

If we assume the Lindel\"of hypothsis for Dedekind zeta function, it holds that for all $\varepsilon>0$, for all $K$ and for all $m$ 
\begin{equation}
\label{lindelof}\Delta_K^m(x)=O_{\varepsilon}\left(x^{\frac12+\varepsilon}D_K^{\varepsilon}\right).\end{equation}

In this paper we estimate the error term of $\Delta_K^m(x)$ by using exponential sums. In \cite{No93} and \cite{gk05}, they use other approaches, so we expect new development for the Piltz divisor problem over number field. As a results, we improve the estimate of upper bound of $\Delta_K^m(x)$ for many $K$ and many $m$.

In Section 2, we show some auxiliary theorems to consider the upper bound of the error term $\Delta_K^m(x)$. 
First we give a review of the convexity bound for the Dedekind zeta function and generalized Atkinson's Lemma \cite{at41}. 
Next we show proposition \ref{ideal}, which reduces an ideal counting problem to an exponential sums problem. This proposition plays a crucial role in our computing $\Delta_K^m(x)$.

In Section 3, we prove the following theorem about the error term $\Delta_K^m(x)$ by using estimate of exponential sums.
\begin{theorem}
For every $\varepsilon>0$ the following estimates hold. When $mn\ge4$, then  \[\Delta_K^m(x)=O_{n,m,\varepsilon}\left(x^{\frac{2mn-3}{2mn+1}+\varepsilon}D_K^{\,\frac{2m}{2mn+1}+\varepsilon}\right).\]
\end{theorem}
This theorem gives improvement of upper bound of $\Delta_K^m(x)$ for $mn\ge4$.

In Section 4, we give some application. First we give an uniform estimate for ideal counting function over number fields. Second we show a good uniform upper bound of the distribution of relatively $r$-prime lattice points over number fields as a corollary of the first application.

In Section 5, we consider a conjecture about estimates for Piltz divisor functions over number field.  It is proposed that for all number fields $K$ and for all $m$ the best upper bound of the error term is better than that on the assumption of the Lindel\"  of Hypothesis (\ref{lindelof}).  If $mn\le3$ this conjecture holds, but the other cases it seems to be very difficult.


\section{Auxiliary Theorem}
In this section, we show some important lemmas for our argument. Let $s=\sigma+it$ and $n=[K:\mathbf{Q}]$. We use the convexity bound of Dedekind zeta function to obtain an upper bound of the error term of Piltz divisor function $\Delta_K^m(x)$. 

It is well-known fact that Dedekind zeta function satisfies the following functional equation:
\begin{equation}
\label{fe}
\zeta_K(1-s)=D_K^{s-\frac 12}2^{n(1-s)}\pi^{-ns}\Gamma(s)^{n}\left(\cos\frac{\pi s}2\right)^{r_1+r_2}\left(\sin\frac{\pi s}2\right)^{r_2}\zeta_K(s), 
\end{equation}
where $r_1$ is the number of real embeddings of $K$ and $r_2$ is the number of pairs of
complex embeddings,

The Phragmen-Lindel\"of principle and (\ref{fe}) give the well-known convexity bound of the Dedekind zeta function \cite{ra59}:
For any $\varepsilon>0$ and $n=[K:\mathbf Q]$
\begin{equation}
\label{convex}
\zeta_K(\sigma+it)=\left\{\begin{array}{ll}
O_{n,\varepsilon}\left(|t|^{\frac n2-n\sigma+\varepsilon}D_K^{\,\frac 12-\sigma+\varepsilon}\right)& \text{ if }\sigma\le0,\\
O_{n,\varepsilon}\left(|t|^{\frac{n(1-\sigma)}2+\varepsilon}D_K^{\,\frac{1-\sigma}2+\varepsilon}\right)&\text{ if }0\le\sigma\le1,\\
O_{n,\varepsilon}\left(|t|^{\varepsilon}D_K^{\varepsilon}\right)&\text{ if } 1\le\sigma
\end{array}\right.
\end{equation}
as $|t|^nD_K\rightarrow\infty$, where $K$ runs through number fields with $[K:\mathbf{Q}]=n$. In the previous papers, we also use this convexity bound (\ref{convex}) to estimate the distribution of ideals. In the following sections, we show some estimate for $\Delta_K^m(x)$ in the similar way to our previous papers.

Lemma \ref{gcs} states the growth of the product of Gamma function and trigonometric functions in the functional equation (\ref{fe}) of Dedekind zeta function. 
\begin{lemma}
\label{gcs}
Let ${\tau}\in\{cos,sin\}$ and $n$ be a positive integer
\begin{align*}
&\frac{\Gamma(s)^{n}}{1-s}\left(\cos\frac{\pi s}2\right)^{r_1+r_2}\left(\sin\frac{\pi s}2\right)^{r_2}\\
=&\hspace{1mm}Cn^{-ns}\Gamma\left(ns-\frac{n+1}2\right){\tau}\left(\frac{n\pi s}2\right)+O_{n}\left(|t|^{-2+n\sigma-\frac n2}\right),
\end{align*}
where $C$ is a constant and $s=\sigma+it$.
\end{lemma}
\begin{proof}
This lemma is shown from the Stirling formula and estimate for trigonometric function. 
\end{proof}

{Next we introduce the generalized Atkinson's lemma.} This lemma is quite useful for calculating integrals of the Dedekind zeta function. 
\begin{lemma}[Atkinson \cite{at41}]
\label{atk}
Let $y>0,\ 1<A\le B$ \text{ and } ${\tau}\in\{\cos, \sin\}$, and we define
\[I=\frac1{2\pi i}\int_{A-iB}^{A+iB}\Gamma(s){\tau}\left(\frac{\pi s}2\right)y^{-s}\ ds.\]
If $y\le B$, then \[I={\tau}(y)+O\left(y^{-\frac12}\min\left(\left(\log \frac By\right)^{-1},B^{\frac12}\right)+y^{-A}B^{A-\frac12}+y^{-\frac12}\right).\]
If $y>B$, then \[I=O\left(y^{-A}\left(B^{A-\frac12}\min\left(\left(\log \frac yB\right)^{-1},B^{\frac12}\right)+A^{A-\frac12}\right)\right).\]
\end{lemma}

Finally we introduce the following lemma to reduce the ideal counting problem to an exponential sum problem.  
\begin{lemma}[Bordell\`es \cite{bo15}]
\label{bo}
Let $1\le L\le R$ be a real number and ${f}$ be an arithmetical function satisfying ${f}(m)=O(m^\varepsilon)$, and let ${\mbox{\boldmath $e$}} (x)=exp(2\pi ix)$ and $F={f}*\mu$, where $*$ is the Dirichlet product symbol. For $a\in\mathbf R-\{1\}$, $b,x\in\mathbf R$ and for every $\varepsilon>0$ the following estimate holds.
\begin{align*}
&\sum_{m\le R}\frac{{f}(m)}{m^{a}}{\tau}\left(2\pi xm^{b}\right)\\
&=O_{n, \varepsilon}\left(\begin{array}{l}
L^{1-a}+R^{\varepsilon}\underset{L<S\le R}{\max}S^{-a}\times\\
\displaystyle \times\underset{S<S_1\le 2S}{\max}\underset{\substack{M,N\le S_1\\MN\asymp S}}{\max}\underset{\substack{M\le M_1\le 2M\\ N\le N_1\le 2N}}{\max}\left|\sum_{M<m\le M_1}F(m)\sum_{N<n\le N_1}{\mbox{\boldmath $e$}}\left(x(mn)^{b}\right)\right|
\end{array}\right).
\end{align*}
\end{lemma}
Next proposition plays a crucial role in our computing $I_K^m(x)$. We consider the distribution of ideals of $\mathcal{O}_K$, where $K$ runs through extensions with $[K:\mathbf Q]=n$ and some conditions. The detail of the conditions will be determined later, but they state the relation of the principal term and the error term.
\begin{proposition}
\label{ideal}
Let $F_K=I_K^m*\mu$.
For every $\varepsilon>0$ the following estimate holds.
\begin{align*}
&\Delta_K^m(x)\\
=&\hspace{1mm}O_{n,m, \varepsilon}\left(\begin{array}{l}L^{1-\alpha}+x^{\frac{mn-1}{2mn}}D_K^{\,\frac1{2n}}R^{\varepsilon}\underset{L\le S\le R}{\max}S^{-\frac{mn+1}{2mn}}\times\\
\displaystyle \times\underset{S<S_1\le 2S}{\max}\underset{\substack{M,N\le S_1\\MN\asymp S}}{\max}\underset{\substack{M\le M_1\le 2M\\ N\le N_1\le 2N}}{\max}\left|\sum_{M<l\le M_1}F_K(m)\sum_{N<k\le N_1}{\mbox{\boldmath $e$}}\left(mn\left(\frac{xlk}{D_K}\right)^{\frac1{mn}}\right)\right|\\
+x^{\frac{mn-2}{2mn}+\varepsilon}D_K^{\,\frac1n+\varepsilon}R^{\frac{mn-2}{2mn}+\varepsilon}+x^{\frac{mn-1}{mn}+\varepsilon}D_K^{\,\frac1n+\varepsilon}R^{-\frac1{mn}+\varepsilon}
\end{array}\right).
\end{align*}
where $K$ runs through number fields with $[K:\mathbf Q]=n$ and some conditions.
\end{proposition}

\begin{proof}
Let $d_K^m(l)$ be the number of $m$-tuples of ideals $(\mathfrak{a}_1, \mathfrak{a}_2,\ldots,\mathfrak{a}_m)$ such that product of their ideal norm $\mathfrak{Na}_1\cdots\mathfrak{Na}_m=l$.
Then one can check it easily that \begin{equation}\label{zeta}
\zeta_K(s)^m=\sum_{l=1}^{\infty}\frac{d_K^m(l)}{l^s}\ \text{ for } \Re s>1 
\end{equation}
and
\[I_K^m(x)=\sum_{l\le x}d_K^m(l).\]
 Thus Perron's formula plays  a crucial role in this proof. 

We consider the integral \[\frac1{2\pi i}\int_{C}\zeta_K(s)^m\frac{x^s}s\ ds,\]
where $C$ is the contour $C_1\cup C_2\cup C_3\cup C_4$ shown in the following Figure \ref{path2}.

\setlength\unitlength{1truecm}
\begin{figure}[h]
\begin{center}
\begin{picture}(4.5,4.5)(0,0)
\small
\put(-1,2){\vector(1,0){4}}
\put(0,0){\vector(0,1){4}}
\put(-0.2,0.5){\vector(1,0){1.1}}
\put(0.9,0,5){\line(1,0){1.1}}
\put(2,0.5){\vector(0,1){2}}
\put(2,2.4){\line(0,1){1.1}}
\put(2,3.5){\vector(-1,0){1.1}}
\put(0.9,3,5){\line(-1,0){1.1}}
\put(0.2,2.1){O}
\put(3.1,2){$\Re(s)$}
\put(-0.3,4.1){$\Im(s)$}
\put(0,2){\circle*{0.1}}
\put(2,3.5){\line(-1,0){0.5}}
\put(-0.2,3.5){\vector(0,-1){1.1}}
\put(-0.2,2.5){\line(0,-1){2}}
\put(-0.1,3.5){\line(1,0){0.2}}
\put(-0.5,3.5){$iT$}
\put(-0.1,0.5){\line(1,0){0.2}}
\put(-0.8,0.3){$-iT$}
\put(2,2.9){\line(0,1){0.2}}
\put(-0.7,2.1){$-\varepsilon$}
\put(2.2,2.2){$1+\varepsilon$}
\put(1.3,3.6){$C_2$}
\put(2.1,1.7){$C_1$}
\put(1.3,0.1){$C_4$}
\put(0.3,1.7){$C_3$}
\end{picture}
\end{center}
\caption{\label{path2}}\end{figure}

In a way similar to the well-known proof of Perron's formula, we estimate
\begin{equation}\label{perron}\frac1{2\pi i}\int_{C_1}\zeta_K(s)^m\frac{x^s}s\ ds=I_K^m(x)+O_{\varepsilon}\left(\frac{x^{1+\varepsilon}}{T}\right).\end{equation}
We can select the large $T$, so that the $O$-term in the right hand side is sufficiently small. For estimating the left hand side by using estimate (\ref{convex}), we divide it into the integrals over $C_2, C_3$ and $C_4$. 

First we consider the integrals over $C_2$ and $C_4$ as
\begin{align*}
&\left|\frac1{2\pi i}\int_{C_2\cup C_4}\zeta_K(s)^m\frac{x^s}s\ ds\right|\\
\le&\frac1{2\pi}\int^{1+\varepsilon}_{-\varepsilon}\left|\zeta_K\left(\sigma+iT\right)\right|^m\frac{x^{\sigma}}{T}\ d\sigma+\frac1{2\pi}\int^{1+\varepsilon}_{-\varepsilon}\left|\zeta_K\left(\sigma-iT\right)\right|^m\frac{x^{\sigma}}{T}\ d\sigma.\\
\end{align*}
It holds by the convexity bound of Dedekind zeta function (\ref{convex}) that their sum is estimated as
\begin{align}
\label{24}
\left|\frac1{2\pi i}\int_{C_2\cup C_4}\zeta_K(s)^m\frac{x^s}s\ ds\right|&=O_{n,m,\varepsilon}\left(\int^{1+\varepsilon}_{-\varepsilon}(T^{mn}D_K^{\,m})^{\frac{1-\sigma}2+\varepsilon}\frac{x^{\sigma}}{T}\ d\sigma\right)\nonumber\\[-3mm]
&{}\\[-3mm]
&=O_{n,m,\varepsilon}\left(\frac{x^{1+\varepsilon}D_K^{\,\varepsilon}}{T^{1-\varepsilon}}+T^{\frac {mn}2-1+\varepsilon}D_K^{\,\frac m2+\varepsilon}x^{-\varepsilon}\right).\nonumber
\end{align}
By the Cauchy residue theorem, (\ref{perron}) and (\ref{24}) we obtain 
\begin{equation}
\label{c3}
\Delta_K^m(x)=\int_{C_3}\zeta_K(s)^m\frac{x^s}s\ ds+O_{n,m,\varepsilon}\left(\frac{x^{1+\varepsilon}D_K^{\,\varepsilon}}{T^{1-\varepsilon}}+T^{\frac {mn}2-1+\varepsilon}D_K^{\,\frac m2+\varepsilon}x^{-\varepsilon}\right).
\end{equation}
Thus it suffices to consider the integral over $C_3$ as
\begin{align*}
\frac1{2\pi i}\int_{C_3}\zeta_K(s)^m\frac{x^s}s\ ds&=\frac 1{2\pi i}\int_{-\varepsilon-iT}^{-\varepsilon+iT}\zeta_K(s)^m\frac{x^{s}}{s}\ ds.\\
\intertext{Changing the variable $s$ to $1-s$, we have}
\frac1{2\pi i}\int_{C_3}\zeta_K(s)^m\frac{x^s}s\ ds&=\frac 1{2\pi i}\int_{1+\varepsilon-iT}^{1+\varepsilon+iT}\zeta_K(1-s)^m\frac{x^{1-s}}{1-s}\ ds.\\
\end{align*} 
From this functional equation  (\ref{fe}), it holds that
\begin{align*}
&\frac1{2\pi i}\int_{C_3}\zeta_K(s)^m\frac{x^s}s\ ds\\
=&\hspace{1mm}\frac 1{2\pi i}\int_{1+\varepsilon-iT}^{1+\varepsilon+iT}\left(D_K^{\,s-\frac12}2^{n(1-s)}\pi^{-ns}\Gamma(s)^{n}\left(\cos\frac{\pi s}2\right)^{r_1+r_2}\left(\sin\frac{\pi s}2\right)^{r_2}\zeta_K(s)\right)^m\frac{x^{1-s}}{1-s}\ ds.\\
\intertext{By lemma \ref{gcs} the integral over $C_3$ can be expressed as}
&\frac1{2\pi i}\int_{C_3}\zeta_K(s)\frac{x^s}s\ ds\\
=&\hspace{1mm}\frac {Cx}{2\pi i}\int_{1+\varepsilon-iT}^{1+\varepsilon+iT}D_K^{\,-\frac m2}\left(\frac{(2n)^{mn}\pi^{mn}x}{D_K^m}\right)^{-s}\Gamma\left(mns-\frac{mn+1}2\right){\tau}\left(\frac{mn\pi s}2\right)\zeta_K(s)\ ds\\
&+O_{n,m,\varepsilon}\left(D_K^{\,\frac m2+\varepsilon}T^{\frac {mn}2-1+\varepsilon}x^{-\varepsilon}\right).\\
\intertext{Changing the variable $mns-\frac{mn+1}2$ to $s$, we have}
&\frac1{2\pi i}\int_{C_3}\zeta_K(s)\frac{x^s}s\ ds\\
=&\hspace{1mm}\frac {Cx^{\frac{mn-1}{2mn}}D_K^{\,\frac1{2n}}}{2\pi i}\int_{\frac{mn-1}2+mn\varepsilon-mniT}^{\frac{mn-1}2+mn\varepsilon+mniT}\left(2mn\pi \left(\frac x{D_K^m}\right)^{\frac1{mn}}\right)^{-s}\Gamma(s){\tau}\left(\frac{\pi s}2+\frac{(mn+1)\pi}4\right)\\
&\times\zeta_K\left(\frac s{mn}+\frac{mn+1}{2mn}\right)\ ds+O_{n,m,\varepsilon}\left(D_K^{\,\frac m2+\varepsilon}T^{\frac {mn}2-1+\varepsilon}x^{-\varepsilon}\right).\\
\intertext{From (\ref{zeta}) the function $\zeta_K(s)^m$ can be expressed as a Dirichlet series. It is absolutely and uniformly convergent on compact subsets on $\Re(s)>1$. Therefore we can interchange the order of summation and integral. Thus we obtain}
&\int\left(2mn\pi \left(\frac x{D_K^m}\right)^{\frac1{mn}}\right)^{-s}\Gamma(s){\tau}\left(\frac{\pi s}2+\frac{(mn+1)\pi}4\right)\zeta_K\left(\frac s{mn}+\frac{mn+1}{2mn}\right)\ ds\\
=&\hspace{1mm}\sum_{l=1}^{\infty}\frac{d_K^m(l)}{l^{\frac{mn+1}{2mn}}}\int\left(2mn\pi \left(\frac {lx}{D_K^m}\right)^{\frac1{mn}}\right)^{-s}\Gamma(s){\tau}\left(\frac{\pi s}2+\frac{(mn+1)\pi}4\right)\ ds,\\
\intertext{where the integration is on the vertical line from $\frac{mn-1}2+mn\varepsilon-mniT$ to $\frac{mn-1}2+mn\varepsilon+mniT$. Properties of trigonometric function lead to \[{\tau}\left(\frac{\pi s}2+\frac{(mn+1)\pi}4\right)=\pm\left\{\begin{array}{ll}
{\tau}\left(\frac{\pi s}2\right)&\text{ if } mn \text{ is odd},\\
\frac1{\sqrt2}\left({\tau}\left(\frac{\pi s}2\right)\pm {\tau_1}\left(\frac{\pi s}2\right)\right)&\text{ if } mn \text{ is even,}
\end{array}
\right.\]
where $\{{\tau},{\tau_1}\}=\{\sin,\cos\}$. Hence it holds that}
&\frac1{2\pi i}\int_{C_3}\zeta_K(s)^m\frac{x^s}s\ ds\\
=&\hspace{1mm}\frac {Cx^{\frac{mn-1}{2mn}}D_K^{\,\frac1{2n}}}{2\pi i}\sum_{l=1}^{\infty}\frac{d_K^m(l)}{l^{\frac{mn+1}{2mn}}}\int_{\frac{mn-1}2+mn\varepsilon-mniT}^{\frac{mn-1}2+mn\varepsilon+mniT}\left(2mn\pi \left(\frac {lx}{D_K^m}\right)^{\frac1{mn}}\right)^{-s}\Gamma(s){\tau}\left(\frac{\pi s}2\right)\ ds\\
&+O_{n,m,\varepsilon}\left(D_K^{\,\frac m2+\varepsilon}T^{\frac {mn}2-1+\varepsilon}x^{-\varepsilon}\right).
\end{align*}
Now we apply lemma \ref{atk} to this integral with $y=2mn\pi\left(\frac{lx}{D_K^m}\right)^{\frac1{mn}},\ A=\frac{mn-1}2+mn\varepsilon,\ B=mnT \text{ and }T=2\pi\left(\frac{xR}{D_K^m}\right)^{\frac1{mn}}$, this becomes
\begin{align*}
&\frac1{2\pi i}\int_{C_3}\zeta_K(s)^m\frac{x^s}s\ ds\\
=&\hspace{1mm}\frac {Cx^{\frac{mn-1}{2mn}}D_K^{\,\frac1{2n}}}{2\pi i}\sum_{l\le R}\frac{d^m_K(l)}{l^{\frac{mn+1}{2mn}}}{\tau}\left(2mn\pi \left(\frac {lx}{D_K^m}\right)^{\frac1{mn}}\right)\\
&+O_{n,m, \varepsilon}\left(x^{\frac{mn-2}{2mn}}D_K^{\,\frac1{n}}\sum_{l\le R}\frac{d_K^m(l)}{l^{\frac{mn+2}{2mn}}}\min\left\{\left(\log \frac Rl\right)^{-1},\ \left(\frac{Rx}{D_K^m}\right)^{\frac1{2mn}}\right\}\right)\\
&+O_{n,m, \varepsilon}\left(x^{\frac{mn-2}{2mn}}D_K^{\,\frac1{n}}\sum_{l\le R}\frac{d_K^m(l)}{l^{\frac{mn+2}{2mn}}}\left(\left(\frac Rl\right)^{\frac{mn-2}{2mn}}+1\right)\right)\\
&+O_{n,m, \varepsilon}\left(x^{\frac{mn-2}{2mn}}D_K^{\,\frac1{n}}R^{\frac{mn-2}{2mn}+\varepsilon}\sum_{l> R}\frac{d_K^m(l)}{l^{1+\varepsilon}}\min\left\{\left(\log \frac lR\right)^{-1},\ \left(\frac{Rx}{D_K^m}\right)^{\frac1{2mn}}\right\}\right)\\
&+O_{n,m, \varepsilon}\left(x^{\frac{mn-2}{2mn}+\varepsilon}D_K^{\,\frac1n+\varepsilon}R^{\frac{mn-2}{2mn}+\varepsilon}\right).
\end{align*}
We evaluate three $O$-terms as follows.

\begin{align*}
\intertext{First we consider the first $O$-term. One can estimate $\left(\log \frac Rl\right)^{-1}=O\left(\frac R{R-l}\right)$, so we obtain}
&O_{n,m, \varepsilon}\left(x^{\frac{mn-2}{2mn}}D_K^{\,\frac1{n}}\sum_{l\le R}\frac{d_K^m(l)}{l^{\frac{mn+2}{2mn}}}\min\left\{\left(\log \frac Rl\right)^{-1},\ \left(\frac{Rx}{D_K^m}\right)^{\frac1{2mn}}\right\}\right)\\
=&\hspace{1mm}O_{n,m, \varepsilon}\left(x^{\frac{mn-2}{2mn}}D_K^{\,\frac1{n}}\sum_{l\le [R]-1}\frac{d_K^m(l)}{l^{\frac{mn+2}{2mn}}}\left(\log \frac Rl\right)^{-1}+x^{\frac{mn-2}{2mn}}D_K^{\,\frac1{n}}\sum_{[R]\le l\le R}\frac{d_K^m(l)}{l^{\frac{mn+2}{2mn}}}\left(\frac{Rx}{D_K^m}\right)^{\frac1{2mn}}\right)\\
=&\hspace{1mm}O_{n,m, \varepsilon}\left(x^{\frac{mn-2}{2mn}}D_K^{\,\frac1{n}}\sum_{l\le [R]-1}\frac{d_K^m(l)}{l^{\frac{mn+2}{2mn}}}\frac R{R-l}+x^{\frac{mn-1}{2mn}}D_K^{\,\frac{1}{2n}}R^{\frac{1}{2mn}}\sum_{[R]\le l\le R}\frac{d_K^m(l)}{l^{\frac{mn+2}{2mn}}}\right)\\
=&\hspace{1mm}O_{n,m, \varepsilon}\left(x^{\frac{mn-2}{2mn}}D_K^{\,\frac1{n}}R^{\frac{mn-2}{2mn}+\varepsilon}+x^{\frac{mn-1}{2mn}}D_K^{\,\frac{1}{2n}}R^{-\frac{mn+1}{2mn}}\right).
\end{align*}

\begin{align*}
\intertext{Next we calculate the second $O$-term.}
O_{n,m, \varepsilon}\left(x^{\frac{mn-2}{2mn}}D_K^{\,\frac1{n}}\sum_{l\le R}\frac{d_K^m(l)}{l^{\frac{mn+2}{2mn}}}\left(\left(\frac Rl\right)^{\frac{mn-2}{2mn}}+1\right)\right)=&\hspace{1mm}O_{n,m, \varepsilon}\left(x^{\frac{mn-2}{2mn}}D_K^{\,\frac1{n}}R^{\frac{mn-2}{2mn}}\sum_{l\le R}\frac{d_K^m(l)}{l}\right).\\
\intertext{Since it is well-known that $d_K^m(l)=O(l^\varepsilon)$, we get}
O_{n,m, \varepsilon}\left(x^{\frac{mn-2}{2mn}}D_K^{\,\frac1{n}}\sum_{l\le R}\frac{d_K^m(l)}{l^{\frac{mn+2}{2mn}}}\left(\left(\frac Rl\right)^{\frac{mn-2}{2mn}}+1\right)\right)=&\hspace{1mm}O_{n,m, \varepsilon}\left(x^{\frac{mn-2}{2mn}}D_K^{\,\frac1{n}}R^{\frac{mn-2}{2mn}}\int_1^R\frac{t^\varepsilon}{t}\ dt\right)\\
=&\hspace{1mm}O_{n,m, \varepsilon}\left(x^{\frac{mn-2}{2mn}}D_K^{\,\frac1{n}}R^{\frac{mn-2}{2mn}+\varepsilon}\right).
\end{align*}

\begin{align*}
\intertext{Finally we estimate the third $O$-term in a similar way to calculate the first $O$-term. One can estimate $\left(\log \frac lR\right)^{-1}=O\left(\frac R{l-R}\right)$, so we obtain}
&O_{n,m, \varepsilon}\left(x^{\frac{mn-2}{2mn}}D_K^{\,\frac1{n}}R^{\frac{mn-2}{2mn}+\varepsilon}\sum_{l> R}\frac{d_K^m(l)}{l^{1+\varepsilon}}\min\left\{\left(\log \frac lR\right)^{-1},\ \left(\frac{Rx}{D_K^m}\right)^{\frac1{2mn}}\right\}\right)\\
=&\hspace{1mm}O_{n,m, \varepsilon}\left(x^{\frac{mn-2}{2mn}}D_K^{\,\frac1{n}}R^{\frac{mn-2}{2mn}+\varepsilon}\left(\sum_{R<l\le[R]+1}\frac{d_K^m(l)}{l^{1+\varepsilon}}\left(\frac{Rx}{D_K^m}\right)^{\frac1{2mn}}+\sum_{[R]+2\le l}\frac{d_K^m(l)}{l^{1+\varepsilon}}\left(\log \frac lR\right)^{-1}\right)\right)\\
=&\hspace{1mm}O_{n,m, \varepsilon}\left(x^{\frac{mn-1}{2mn}}D_K^{\,\frac1{2n}}R^{\frac{mn-1}{2mn}+\varepsilon}\sum_{R<l\le[R]+1}\frac{d_K^m(l)}{l^{1+\varepsilon}}+x^{\frac{mn-2}{2mn}}D_K^{\,\frac1{n}}R^{\frac{mn-2}{2mn}+\varepsilon}\sum_{[R]+2\le l}\frac{d_K^m(l)}{l^{1+\varepsilon}}\frac R{l-R}\right)\\
=&\hspace{1mm}O_{n,m, \varepsilon}\left(x^{\frac{mn-1}{2mn}}D_K^{\,\frac1{2n}}R^{-\frac{mn+1}{2mn}+\varepsilon}+x^{\frac{mn-2}{2mn}}D_K^{\,\frac1{n}}R^{\frac{mn-2}{2mn}+\varepsilon}\right).
\end{align*}
From above results, we obtain
\begin{align}
\label{c32}
\frac1{2\pi i}\int_{C_3}\zeta_K(s)^m\frac{x^s}s\ ds=&\hspace{1mm}\frac {Cx^{\frac{mn-1}{2mn}}D_K^{\,\frac1{2n}}}{2\pi i}\sum_{l\le R}\frac{d_K^m(l)}{l^{\frac{mn+1}{2mn}}}{\tau}\left(2n\pi \left(\frac {lx}{D_K^m}\right)^{\frac1{mn}}\right)\nonumber\\[-3mm]
&{}\\[-3mm]
&+O_{n,m, \varepsilon}\left(x^{\frac{mn-1}{2mn}}D_K^{\,\frac1{2n}}R^{-\frac{mn+1}{2mn}+\varepsilon}+x^{\frac{mn-2}{2mn}}D_K^{\,\frac1{n}}R^{\frac{mn-2}{2mn}+\varepsilon}\right).\nonumber
\end{align}

\begin{align*}
\intertext{From estimate (\ref{c3}) and (\ref{c32}), it is obtained that}
\Delta_K^m(x)=&\hspace{1mm}\frac {Cx^{\frac{mn-1}{2mn}}D_K^{\,\frac1{2n}}}{2\pi i}\sum_{l\le R}\frac{d_K^m(l)}{l^{\frac{mn+1}{2mn}}}{\tau}\left(2mn\pi \left(\frac {lx}{D_K^m}\right)^{\frac1{mn}}\right)\\
&+O_{n,m,\varepsilon}\left(x^{\frac{mn-2}{2mn}+\varepsilon}D_K^{\,\frac1n+\varepsilon}R^{\frac{mn-2}{2mn}+\varepsilon}+x^{\frac{mn-1}{mn}+\varepsilon}D_K^{\,\frac1n+\varepsilon}R^{-\frac1{mn}+\varepsilon}\right).\\
\end{align*}
Next we consider the above sum. Let $F_K=d^m_K*\mu$, where $*$ is the Dirichlet product symbol. From lemma \ref{bo}  this becomes 
\begin{align*}
&\Delta_K^m(x)\\
=&\hspace{1mm}O_{n,m, \varepsilon}\left(\begin{array}{l}L^{1-\alpha}+x^{\frac{mn-1}{2mn}}D_K^{\,\frac1{2n}}R^{\varepsilon}\underset{L\le S\le R}{\max}S^{-\frac{mn+1}{2mn}}\times\\
\displaystyle \times\underset{S<S_1\le 2S}{\max}\underset{\substack{M,N\le S_1\\MN\asymp S}}{\max}\underset{\substack{M\le M_1\le 2M\\ N\le N_1\le 2N}}{\max}\left|\sum_{M<l\le M_1}F_K(l)\sum_{N<k\le N_1}{\mbox{\boldmath $e$}}\left(mn\left(\frac{xlk}{D_K^m}\right)^{\frac1{mn}}\right)\right|\\
+x^{\frac{mn-2}{2mn}+\varepsilon}D_K^{\,\frac1n+\varepsilon}R^{\frac{mn-2}{2mn}+\varepsilon}+x^{\frac{mn-1}{mn}+\varepsilon}D_K^{\,\frac1n+\varepsilon}R^{-\frac1{mn}+\varepsilon}
\end{array}\right).
\end{align*}
This proves this proposition.
\end{proof}
Let $\mathcal S_K(x,S)$ be the sum in the $O$-term, that is, \[\underset{L\le S\le R}{\max}S^{-\frac{mn+1}{2mn}}\underset{S<S_1\le 2S}{\max}\underset{\substack{M,N\le S_1\\MN\asymp S}}{\max}\underset{\substack{M\le M_1\le 2M\\ N\le N_1\le 2N}}{\max}\left|\sum_{M<l\le M_1}F_K(l)\sum_{N<k\le N_1}{\mbox{\boldmath $e$}}\left(mn\left(\frac{xlk}{D_K^m}\right)^{\frac1{mn}}\right)\right|.\]
This proposition reduces the initial problem to an exponential sums problem. There are many results to estimate exponential sum. In the next section, we estimate Piltz divisor function by using some results for exponential sum established by many authors.


\section{Estimate of counting function}
In the last section, we show that the error term of Piltz divisor function $\Delta_K^m(x)$ can be expressed as a exponential sum. Let $X > 1$ be a real number, $1\le M < M_1 \le 2M$ and $1\le N<N_1\le 2N$ be integers and $(a_m), (b_n)  \subset \mathbf C$ be sequence of complex numbers, and let $\alpha,\beta\in \mathbf R$ and we define
\begin{equation}\label{expo} 
\mathcal S=\sum_{M<m\le M_1}a_m\sum_{N<n\le N_1}b_n\mbox{\boldmath $e$}\left(X\left(\frac mM\right)^{\alpha}\left(\frac nN\right)^{\beta}\right).
\end{equation}
In 1998 Wu shows this lemma.
\begin{lemma}[Wu \cite{wu98}]
\label{54}
Let $\alpha,\beta\in \mathbf R$ such that $\alpha\beta(\alpha-1)(\beta-1)\not=0$, and $|a_m|\le 1$ and $|b_n| \le 1$ and $\mathcal L=log(XMN + 2)$. Then
\[\mathcal L^{-2}\mathcal S=\hspace{1mm}O\left(\begin{array}{l}
(XM^3N^4)^{\frac15}+(X^4M^{10}N^{11})^{\frac1{16}}+(XM^7N^{10})^{\frac1{11}}\\
+MN^{\frac12}+(X^{-1}M^{14}N^{23})^{\frac1{22}} + X^{-\frac12}MN
\end{array}\right).
\]
\end{lemma}
Next Bordell\`es also shows this lemma by using estimate for triple exponential sums by Robert and Sargos.
\begin{lemma}[Bordell\`es \cite{bo15}]
\label{55}
Let $\alpha,\beta\in \mathbf R$ such that $\alpha\beta(\alpha-1)(\beta-1)\not=0$, and $|a_m|\le 1$ and $|b_n| \le 1$. If $X=O(M)$ then
\begin{align*}
&(MN)^{-\varepsilon}\mathcal S\\
=&\hspace{1mm}O\left((XM^5N^7)^{\frac18}+N(X^{-2}M^{11})^{\frac1{12}}+(X^{-3}M^{21}N^{23})^{\frac1{24}}+M^{\frac34}N + X^{-\frac14}MN\right).
\end{align*}
\end{lemma}
The following {{Srinivasan}}'s result is important for our estimating $\Delta_K^m(x)$.
\begin{lemma}[{{Srinivasan}} \cite{sr62}]
\label{srr}
Let $N$ and $P$ be positive integers and $u_n\ge0$, $v_p>0$, $A_n$ and $B_p$ denote constants for $1\le n\le N$ and $1\le p\le P$. Then there exists $q$ with properties \[Q_1\le q\le Q_2\]
and \[\sum_{n=1}^NA_nq^{u_n}+\sum_{p=1}^PB_pq^{-v_p}=O\left(\sum_{n=1}^N\sum_{p=1}^P\sqrt[u_n+v_p]{A_n^{v_p}B_p^{u_n}}+\sum_{n=1}^NA_nQ_1^{u_n}+\sum_{p=1}^PB_pQ_2^{-v_p}\right).\]
The constant involved in $O$-symbol is less than $N+P$.  
\end{lemma}
{{Srinivasan}} remarks that the inequality in lemma \ref{srr} corresponds to the `best possible' choice of $q$ in the range $Q_1\le q\le Q_2$ \cite{sr62}. We apply lemma \ref{srr} to improve the error term $\Delta_K^m(x)$. 

\begin{theorem}
\label{cub}
For every $\varepsilon>0$ the following estimates hold.
When $mn\ge4$, then  \[\Delta_K^m(x)=O_{n,m,\varepsilon}\left(x^{\frac{2mn-3}{2mn+1}+\varepsilon}D_K^{\,\frac{2m}{2mn+1}+\varepsilon}\right)\]
as $x$ tends to infinity.
\end{theorem}
\begin{proof}
We note that 
\begin{align*}
&\left|\sum_{M<l\le M_1}F_K(l)\sum_{N<k\le N_1}{\mbox{\boldmath $e$}}\left(mn\left(\frac{xlk}{D_K^m}\right)^{\frac1{mn}}\right)\right|\\
=&\hspace{1mm}\left|\sum_{M<l\le M_1}F_K(l)\sum_{N<k\le N_1}{\mbox{\boldmath $e$}}\left(mn\left(\frac{xMN}{D_K^m}\right)^{\frac1{mn}}\left(\frac lM\right)^{\frac1{mn}}\left(\frac kN\right)^{\frac1{mn}}\right)\right|.
\end{align*}
We use the above lemmas with $X=mn\left(\frac{xMN}{D_K^m}\right)^{\frac1{mn}}>0$.
Let $0\le\alpha\le\frac13$, we consider four cases:

\begin{center}
\begin{tabular}{cl}
\\
\hline
Case 1. &$S^{\alpha}\ll N\ll S^{\frac12}$\\
Case 2. &$S^{\frac12}\ll N\ll S^{1-\alpha}$\\
Case 3. &$S^{1-\alpha}\ll N$\\
Case 4. &$N\ll S^{\alpha}$
\end{tabular}
\end{center}

When $S^{\alpha}\ll N\ll S^{\frac12}$, we apply lemma \ref{54} and this gives 
\begin{align}\label{case1}
&S^{-\varepsilon}x^{\frac{mn-1}{2mn}}D_K^{\,\frac1{2n}}\mathcal S_K(x,S)\nonumber\\[-3mm]
&{}\\[-3mm]
=&\hspace{1mm}O_{n,m, \varepsilon}\left(\begin{array}{l}
x^{\frac{5mn-3}{10mn}}D_K^{\,\frac3{10n}}R^{\frac{2mn-3}{10mn}}+x^{\frac{2mn-1}{4mn}}D_K^{\,\frac1{4n}}R^{\frac{5mn-8}{32mn}}\\
+x^{\frac{11mn-9}{22mn}}D_K^{\,\frac9{22n}}R^{\frac{6mn-9}{22mn}}+x^{\frac{mn-1}{2mn}}D_K^{\,\frac1{2n}}R^{\frac{mn-1}{2mn}-\frac12\alpha}\\
+x^{\frac{11mn-12}{22mn}}D_K^{\,\frac6{11n}}R^{\frac{15mn-24}{44mn}}+x^{\frac{mn-2}{2mn}}D_K^{\,\frac1{n}}R^{\frac{mn-2}{2mn}}
\end{array}\right).\nonumber
\end{align}
When  $S^{\frac12}\ll N\ll S^{1-\alpha}$ we use lemma \ref{54} again reversing the role of $M$ and $N$. We obtain the same estimate for the case that $S^{\alpha}\ll N\ll S^{\frac12}$.

\noindent
For the case 3, we use lemma \ref{55}
\begin{align}
\label{case3}
&S^{-\varepsilon}x^{\frac{mn-1}{2mn}}D_K^{\,\frac1{2n}}\mathcal S_K(x,S)\nonumber\\[-3mm]
&{}\\[-3mm]
=&\hspace{1mm}O_{n,m, \varepsilon}\left(\begin{array}{l}
x^{\frac{4mn-3}{8mn}}D_K^{\,\frac3{8n}}R^{\frac{mn-3}{8mn}+\frac14\alpha}+x^{\frac{3mn-4}{6mn}}D_K^{\,\frac2{3n}}R^{\frac{5mn-8}{12mn}+\frac1{12}\alpha}\\
+x^{\frac{4mn-5}{8mn}}D_K^{\,\frac5{8n}}R^{\frac{3mn-5}{8mn}-\frac1{12}\alpha}\\
+x^{\frac{mn-1}{2mn}}D_K^{\,\frac1{2n}}R^{\frac{mn-2}{4mn}+\frac14\alpha}+x^{\frac{2mn-3}{4mn}}D_K^{\,\frac3{4n}}R^{\frac{2mn-3}{4mn}}
\end{array}\right).\nonumber
\end{align}
If $x^{\frac1{mn(1-\alpha)-1}}D_K^{-\frac m{mn(1-\alpha)-1}}\ll S$,  the condition of Lemma \ref{55} $X=O(N)$ is satisfied. Therefore it suffices to choose $L=x^{\frac1{mn(1-\alpha)-1}}D_K^{-\frac m{mn(1-\alpha)-1}}$.
For the case 4, we use Lemma \ref{55} again reversing the role of $M$ and $N$. We obtain the same estimate for the case that $N\ll S^{\alpha}$.
Combining (\ref{case1}) and (\ref{case3}) with proposition \ref{ideal}, we obtain
\begin{equation}
\label{ue}
\Delta_K^m(x)=O_{n,m, \varepsilon}\left(\begin{array}{l}
x^{\frac{5mn-3}{10mn}}D_K^{\,\frac3{10n}}R^{\frac{2mn-3}{10mn}+\varepsilon}+x^{\frac{2mn-1}{4mn}}D_K^{\,\frac1{4n}}R^{\frac{5mn-8}{32mn}+\varepsilon}\\
+x^{\frac{11mn-9}{22mn}}D_K^{\,\frac9{22n}}R^{\frac{6mn-9}{22mn}+\varepsilon}+x^{\frac{mn-1}{2mn}}D_K^{\,\frac1{2n}}R^{\frac{mn-1}{2mn}-\frac12\alpha+\varepsilon}\\
+x^{\frac{11mn-12}{22mn}}D_K^{\,\frac6{11n}}R^{\frac{15mn-24}{44mn}+\varepsilon}+x^{\frac{mn-2}{2mn}}D_K^{\,\frac1{n}}R^{\frac{mn-2}{2mn}+\varepsilon}\\
+x^{\frac{4mn-3}{8mn}}D_K^{\,\frac3{8n}}R^{\frac{mn-3}{8mn}+\frac14\alpha+\varepsilon}+x^{\frac{3mn-4}{6mn}}D_K^{\,\frac2{3n}}R^{\frac{5mn-8}{12mn}+\frac1{12}\alpha+\varepsilon}\\
+x^{\frac{4mn-5}{8mn}}D_K^{\,\frac5{8n}}R^{\frac{3mn-5}{8mn}+\frac1{12}\alpha+\varepsilon}+x^{\frac{2mn-3}{4mn}}D_K^{\,\frac3{4n}}R^{\frac{2mn-3}{4mn}+\varepsilon}\\
+x^{\frac{mn-1}{mn}+\varepsilon}D_K^{\,\frac1n+\varepsilon}R^{-\frac1{mn}+\varepsilon}+x^{\frac{1-\alpha}{mn(1-\alpha)-1}}D_K^{-\frac {m(1-\alpha)}{mn(1-\alpha)-1}}
\end{array}\right).
\end{equation}
By lemma \ref{srr} with $x^{\frac1{mn(1-\alpha)-1}}D_K^{-\frac m{mn(1-\alpha)-1}}\le R\le xD$ there exists $R$ such that the error term of estimate (\ref{ue}) is much less than
\[\begin{array}{l}
x^{\frac{2mn}{2mn+7}+\varepsilon}D_K^{\,\frac{2m}{2mn+7}+\varepsilon}+x^{\frac{5mn+3}{5mn+24}+\varepsilon}D_K^{\,\frac{5m}{5mn+24}+\varepsilon}+x^{\frac{6mn-4}{6mn+13}+\varepsilon}D_K^{\,\frac{6m}{6mn+13}+\varepsilon}\\
+x^{\frac{(1-\alpha)mn+\alpha-1}{(1-\alpha)mn+1}+\varepsilon}D_K^{\,\frac{(1-\alpha)m}{(1-\alpha)mn+1}+\varepsilon}
+x^{\frac{15mn-17}{15mn+20}+\varepsilon}D_K^{\,\frac{3m}{3mn+4}+\varepsilon}+x^{\frac{mn-2}{mn}+\varepsilon}D_K^{\,\frac {1}{n}+\varepsilon}\\+x^{\frac{(2\alpha+1)mn-2\alpha}{(2\alpha+1)mn+5}+\varepsilon}D_K^{\,\frac{(2\alpha+1)m}{(2\alpha+1)mn+5}+\varepsilon}+x^{\frac{(\alpha+5)mn-\alpha-7}{(\alpha+5)mn+4}+\varepsilon}D_K^{\,\frac{(\alpha+5)m}{(\alpha+5)mn+4}+\varepsilon}\\
+x^{\frac{(2\alpha+9)mn-2\alpha-12}{(2\alpha+9)mn+9}+\varepsilon}D_K^{\,\frac{(2\alpha+9)m}{(2\alpha+9)mn+9}+\varepsilon}+x^{\frac{2mn-3}{2mn+1}+\varepsilon}D_K^{\,\frac{2m}{2mn+1}+\varepsilon}\\
+x^{\frac{5mn(1-\alpha)-6+3\alpha}{10mn(1-\alpha)-10}+\varepsilon}D_K^{\,\frac{m-3m\alpha}{10mn(1-\alpha)-10}+\varepsilon}+x^{\frac{16mn(1-\alpha)-19+8\alpha}{32mn(1-\alpha)-32}+\varepsilon}D_K^{\,\frac{3m-8m\alpha}{32mn(1-\alpha)-32}+\varepsilon}\\
+x^{\frac{11mn(1-\alpha)-14+9\alpha}{22mn(1-\alpha)-22}+\varepsilon}D_K^{\,\frac{3m-9m\alpha}{22mn(1-\alpha)-22}+\varepsilon}+x^{\frac{1}{2}+\varepsilon}\\
+x^{\frac{22mn(1-\alpha)-31+24\alpha}{44mn(1-\alpha)-44}+\varepsilon}D_K^{\,\frac{9m-24m\alpha}{44mn(1-\alpha)-44}+\varepsilon}+x^{\frac{mn(1-\alpha)-2+2\alpha}{2mn(1-\alpha)-2}+\varepsilon}D_K^{\,\frac{m-2m\alpha}{2mn(1-\alpha)-2}+\varepsilon}\\
+x^{\frac{4mn(1-\alpha)-6+5\alpha}{8mn(1-\alpha)-8}+\varepsilon}D_K^{\,\frac{2m-5m\alpha}{8mn(1-\alpha)-8}+\varepsilon}+x^{\frac{6mn(1-\alpha)-9+9\alpha}{12mn(1-\alpha)-12}+\varepsilon}D_K^{\,\frac{3m-9m\alpha}{12mn(1-\alpha)-12}+\varepsilon}\\
+x^{\frac{12mn(1-\alpha)-18+17\alpha}{24mn(1-\alpha)-24}+\varepsilon}D_K^{\,\frac{6m-17m\alpha}{24mn(1-\alpha)-24}+\varepsilon}+x^{\frac{2mn(1-\alpha)-3+3\alpha}{4mn(1-\alpha)-4}+\varepsilon}D_K^{\,\frac{m-3m\alpha}{4mn(1-\alpha)-4}+\varepsilon}\\
+x^{\frac{1-\alpha}{mn(1-\alpha)-1}}D_K^{-\frac {m(1-\alpha)}{mn(1-\alpha)-1}}.
\end{array}\]
When $mn\ge4$ and $\alpha=\frac{mn+3}{7mn-5}$, then we have 
\[\Delta_K^m(x)=O_{n,m,\varepsilon}\left(x^{\frac{2mn-3}{2mn+1}+\varepsilon}D_K^{\,\frac{2m}{2mn+1}+\varepsilon}\right).\]
This proves the theorem. 
\end{proof}
For $mn\ge4$ this theorem gives new results for Piltz divisor problem over number field. In particular, if we {fix} $K$ with $[K:\mathbf Q]=4$ then we improve the estimate for $\Delta_K(x)$ as follows:
\begin{corollary}
For any number field $K$ with $[K:\mathbf Q]=4$, \[\Delta_K(x)=O_{K,\varepsilon}\left(x^{\frac59+\varepsilon}\right).\]
\end{corollary}
This result is better than Bordell\`es' result.


\section{Application}
In this section we introduce some applications of our theorems. First we obtain uniform estimate for ideal counting function $I_K(x)$. From the proof of theorem \ref{cub}, we obtain the following theorem.
\begin{theorem}
\label{cubi}
For all $\varepsilon> 0$ for any fixed $0\le\beta\le\frac8{2n+5}-\varepsilon$ and $C>0$ the followings hold.
If $K$ runs through number fields with $[K:\mathbf Q]\le n$ and $D_K\le Cx^{\beta}$ then
\[\Delta_K(x)=O_{C,n,\varepsilon}\left(x^{\frac{2n-3+2\beta}{2n+1}+\varepsilon}\right).\]
\end{theorem}
The condition $D_K\le Cx^{\beta}$ is caused by the relation between the principal term and the error term.
 It is well known that $I_K(x)$ is very important to estimate the distribution of relatively $r$-prime lattice points. We regard an $\ell$-tuple of ideals $(\mathfrak{a}_1, \mathfrak{a}_2,\ldots,\mathfrak{a}_\ell)$ of $\mathcal{O}_K$ as a lattice point in $K^\ell$.  We say that a lattice point $(\mathfrak{a}_1, \mathfrak{a}_2,\ldots,\mathfrak{a}_\ell)$ is {\it relatively $r$-prime} for a positive integer $r$, if there exists no prime ideal $\mathfrak{p}$ such that $\mathfrak{a}_1, \mathfrak{a}_2,\ldots,\mathfrak{a}_\ell\subset \mathfrak{p}^r$. Let $V^r_\ell(x,K)$ denote the number of relatively $r$-prime lattice points $(\mathfrak{a}_1, \mathfrak{a}_2,\ldots,\mathfrak{a}_\ell)$ such that their ideal norm $\mathfrak{Na}_i\le x$.

B. D. Sittinger shows that \[V^r_\ell(x,K)\sim\frac{\rho_K^\ell}{\zeta_K(r\ell)}x^\ell,\]
where $\rho_K$ is the residue of $\zeta_K$ as $s=1$ \cite{St10}.
{It is well known that
\begin{equation}
\rho_K=\frac{2^{r_1}(2\pi)^{r_2}h_KR_K}{w_K\sqrt{D_K}},\label{crho}
\end{equation}
where $h_K$ is the class number of $K$, $R_K$ is the regulator of $K$ and $w_K$ is the number of roots of unity in $\mathcal{O}^*_K$.}

After that we show some results for the error term:\[E_\ell^r(x,K)=V_\ell^r(x,K)-\frac{\rho_K^\ell}{\zeta_K(r\ell)}x^\ell.\]
In \cite{Ta16} and \cite{tk17} we consider the relation between relatively $r$-prime problem and other mathematical problems. 
If we assume the Lindel\"{o}f Hypothesis for $\zeta_K(s)$, then it holds that for all $\varepsilon> 0$
\begin{equation}
E_\ell^r(x,K)=\left\{
\begin{array}{ll}
O_{\varepsilon}\left(x^{\frac1r(\frac32+\varepsilon)}\right)&\text{ if } r\ell=2,\\
O_{\varepsilon}\left(x^{\ell-\frac12+\varepsilon}\right)&\text{ otherwise }
\end{array}
\right.
\end{equation}
From easy calculation, we obtain the following corollary.
\begin{corollary}
For all $\varepsilon> 0$ and for any fixed $0\le\beta\le\frac8{2n+5}-\varepsilon$ and $C>0$ the followings hold.
If $K$ runs through number fields with $[K:\mathbf Q]\le n$ and $D_K\le Cx^{\beta}$, then \[E_\ell^r(x,K)=\left\{\begin{array}{ll}
O_{C, n,\varepsilon}\left(x^{\frac{4n-2}{r(2n+1)}+\frac{4}{2n+1}\beta+\varepsilon}\right)& \text{ if } r\ell=2,\\
O_{C,n,\varepsilon}\left(x^{\ell-\frac {4}{2n+1}+\frac{2n+5-(2n+1)\ell}{2(2n+1)}\beta+\varepsilon}\right)& \text{ otherwise. }
\end{array}\right.\]
\end{corollary}
For the proof of this corollary, please see the proof of Theorem 4.1 of \cite{tk17}.

\section{Conjecture}
Theorem \ref{cubi} states good uniform upper bounds. It is proposed that for all number fields $K$ the best uniform upper bound of the error term is better than that on the assumption of the Lindel\"  of Hypothesis (\ref{lindelof}). 

\begin{conj}
If $K$ runs through number fields with $D_K<x$, then 
\[\Delta_K^m(x)=o\left(x^{\frac12}\right).\]
\end{conj}

If $K$ runs through cubic extension fields with $D_K\le Cx^{\frac14-\varepsilon}$, then this conjecture holds from theorem \ref{cubi}.

From estimate (\ref{omega}), this conjecture may give the best estimate for uniform upper bound of $\Delta_K^m(x)$. As we remarked above (Theorem \ref{idealhi}) this conjecture is very difficult even when $K$ is fixed and $m=1$.

\end{document}